\documentclass{article}
\usepackage[utf8]{inputenc}

\usepackage{amsthm}
\usepackage{amsmath}
\usepackage{amssymb}
\usepackage{url}
\usepackage{float}
\usepackage{bm}

\usepackage{graphicx}
\usepackage{mathtools}

\usepackage{hyperref}
    \hypersetup{
    colorlinks=true,%
    citecolor=blue,%
    filecolor=black,%
    linkcolor=blue,%
    urlcolor=black,%
    }
\usepackage{verbatim}
\usepackage{listings}
\usepackage{matlab-prettifier}
\usepackage{todonotes}

\DeclareMathOperator*{\argmax}{arg\,max}
\DeclareMathOperator*{\argmin}{arg\,min}

\newcommand{\R}{\mathbb{R}}

\usepackage[ruled,vlined]{algorithm2e}

\newcommand{\prox}{\operatorname{prox}}
\newcommand{\Id}{\operatorname{Id}}

\newtheorem{theorem}{Theorem}[section]

\newtheorem{proposition}{Proposition}[section]
\theoremstyle{remark}

\theoremstyle{definition}

\newtheorem{lemma}[theorem]{Lemma}
\usepackage{graphicx}
\usepackage{fancyhdr} 
\usepackage[top=3cm,bottom=3cm,left=3.2cm,right=3.2cm,headsep=10pt]{geometry}
\usepackage[ruled,vlined]{algorithm2e}
\numberwithin{equation}{section}
\title{Proximity operator of the weighted squared $\ell_{2,\infty}$ norm with applications}
\author{
Sergio López-Rivera\footnote{Departamento de Matemática, Universidad Tecnica Federico Santa Maria, Santiago, Chile.
		E-mail: \href{mailto:sergio.lopez@dim.uchile.cl}{sergio.lopezr@usm.cl}}}
\date{\today}

\makeatletter
\newcommand{\otherlabel}[2]{\protected@edef\@currentlabel{#2}\label{#1}}
\makeatother

\begin{document}

 \maketitle
 \begin{abstract}
In this paper, we derive a closed-form for the proximity operator of the weighted squared $\ell_{2,\infty}$ norm. Moreover, we derive a closed-form for the proximity operator of a weakly convex version of this function. The proof involves using known results for compute the proximity operator of a supremum function, which requires finding a solution of an auxiliary problem. We find the explicit solution of this auxiliary problem by finding the KKT multipliers and the critical point associated to the first-order optimality conditions. Additionally, we present applications to denoising problems and weighted max-min dispersion problems.
 
 \end{abstract}

	\section{Introduction}
    The proximity operator (also called proximal mapping) is essential in convex analysis and optimization. It appears in the called proximal splitting algorithms \cite{rockafellar1976monotone, vu2013splitting, chambolle2011first, lions1979splitting}, which are used to solve optimization problems that have applications in several fields of engineering and applied mathematics as image processing, denoising, stochastic traffic theory, LASSO problems, evolution inclusions, variational inequalities, machine learning, partial differential equations, mean field games, among others \cite{attouch2016strongly, briceno2023primal,facchinei2003finite, gabay1983chapter,mercier1979lectures, attouch2008alternating}. In this paper, we compute the proximity operator of the squared $\ell_{2,\infty}$ norm with weights, as well as of the weakly convex version of this function. In the case when every weight $\omega_{i}$ is equal to $1$, we recover the results in \cite{briceno2025dro} and \cite{lopez2026projected} for the convex case and weakly convex case, respectively. In addition, we provide an application to a denoising problem for the convex case and to max-min dispersion problems with weights for the weakly convex case.
    
\textbf{Notation:} The euclidean norm of a vector $x\in\R^{n}$ is denoted by $\|x\|$ and the inner product is denoted by $\langle x,y\rangle$ for all $x,y\in\R^{n}$. We denote $\R_{+}=\left[0,+\infty\right[$ and $\R_{++}=\left]0,+\infty\right[$. The probability simplex is defined by $\Delta_{N}:=\{p\in\R_{+}^{N}\,:\,\sum_{i=1}^{N}p_{i}=1\}$. The set of the convex, lower semicontinuous and proper functions from $\R^{n}$ to $\R\cup\{+\infty\}$ is denoted by $\Gamma_{0}(\R^{n})$. Given $f\colon\R^{n}\rightarrow\R\cup\{+\infty\}$ and $\lambda>0$ the Moreau envelope is
\begin{align}
\label{def_mor_env}
(\forall x\in\R^{n})\,\,e_{\lambda}f(x)=\inf_{y\in\R^{n}}\left\{f(y)+\dfrac{1}{2\lambda}\|y-x\|^{2}\right\}.
\end{align}
If $f\in\Gamma_{0}(\R^{n})$, then the infimum in \eqref{def_mor_env} is attained in an unique point, denoted by $\prox_{\lambda f}(x)$ and the operator $x\mapsto \prox_{\lambda f}(x)$ is called the proximity operator of $f$. Given $\rho\geq 0$, the set of the $\rho$-weakly convex functions is
\begin{align*}
\Gamma_{\rho}(\R^{n})=\left\{f\colon\R^{n}\rightarrow\R\cup\{+\infty\}\,:\,f+\dfrac{\rho}{2}\|\cdot\|^{2}\in\Gamma_{0}(\R^{n})\right\}.
\end{align*}
If $f\in\Gamma_{\rho}(\R^{n})$ and $\lambda<1/\rho$, then the proximity operator of $f$ with parameter $\lambda$ is well defined \cite[Theorem~3.4(d)]{attouch1993approximation}. For a nonempty closed convex set $S\subset \R^{n}$, the distance function and the projection mapping are defined as follows:
$$
d(x,S):=\inf_{y\in S}\Vert x-y\Vert \textrm{ and } P_{S}(x)=\displaystyle\argmin_{y\in S}\|x-y\|.
$$
The (convex) subdifferential of a function $f\in\Gamma_{0}(\R^{n})$ is denoted by $\partial f$. If $f\in\Gamma_{\rho}(\R^{n})$, the subdifferential of $f$ is
\begin{align*}
(\forall x\in\R^{n})\quad\partial f(x)=\partial\left(f+\frac{\rho}{2}\|\cdot\|^{2}\right)-\rho x,
\end{align*}
 where the subdifferential on the
right-hand side is the convex subdifferential.

\section{Main results}
\subsection{Convex case}
In this section, we provide a closed-form for the proximity operator of the following function
\begin{align}
\label{fun_conv_1}
(\forall \bm{x}\in\R^{nN})\quad f(\bm{x})=\max_{1\leq i\leq N}\omega_{i}\|x_{i}\|^{2},
\end{align}
where, for all $i\in\{1,\ldots,N\}$, $\omega_{i}>0$.
\begin{proposition}
\label{prox_conv_1}
Let $f$ the function in \eqref{fun_conv_1}. Let $\bm{x}\in\R^{nN}$ and $\lambda>0$. For every $i\in\{1,\ldots,N\}$, denote $\lambda_{i}=\lambda\omega_{i}$. Then
\begin{align*}
\prox_{\lambda f}(\bm{x})=\left(\dfrac{x_{i}}{1+2\lambda_{i}\overline{p}_{i}}\right)_{i=1}^{N},
\end{align*}
where $\overline{p}\in\R^{N}$ is a solution to
\begin{align}
\label{prob_aux_1}
\max_{p\in\Delta_{N}}\ell(p):=\sum_{i=1}^{N}\dfrac{\alpha_{i}p_{i}}{1+2\lambda_{i}p_{i}},
\end{align}
and, for all $i\in\{1,\ldots,N\}$, $\alpha_{i}=\omega_{i}\|x_{i}\|^{2}$. 
\end{proposition}
\begin{proof}
Note that $f(\bm{x})=\displaystyle\max_{p\in\Delta_{N}}g_{p}(\bm{x})$, where $g_{p}(\bm{x})=\sum_{i=1}^{N}p_{i}\omega_{i}\|x_{i}\|^{2}$ for all $p\in\Delta_{N}$. Clearly, $g_{p}\in\Gamma_{0}(\R^{nN})$ for all $p\in\Delta_{N}$. Moreover, the function $p\mapsto g_{p}(\bm{x})$ is concave and upper semicontinuous for all $\bm{x}\in\R^{nN}$. In addition, $f=\sup_{p\in\Delta_{N}}g_{p}$ is proper. Furthermore, $\Delta_{N}$ is a nonempty compact and convex set.
Then, by \cite[Theorem~3.5]{MR4279933}, we have that 
  \begin{align}
   \label{prox_sup_sep}
\prox_{\lambda f}(\bm{x})=\prox_{\lambda g_{\overline{p}}}(\bm{x}),\quad\text{ with } \overline{p}\in \displaystyle\argmax_{p\in\Delta_{N}}e_{\lambda}g_{p}(\bm{x}).
  \end{align}
  Now, by \cite[Proposition~24.11 \& Proposition~24.8(i)]{MR3616647}, we obtain that
 \begin{align}
 \label{prox_gp_1}
\prox_{\lambda g_{\overline{p}}}(\bm{x})=\left(\frac{1}{2\lambda_{i} \overline{p}_{i}+1}x_{i}\right)_{i=1}^{N}.
 \end{align}
 Let us calculate $e_{\lambda}g_{p}(\bm{x})$. By \eqref{prox_gp_1}, we have that
\begin{align*}
e_{\lambda}g_{p}(\bm{x})&=g_{p}(\prox_{\lambda g_{p}}(\bm{x}))+\frac{1}{2\lambda}\|\bm{x}-\prox_{\lambda g_{p}}(\bm{x})\|^{2}\\
&=\sum_{i=1}^{N}p_{i}\omega_{i}\left\|\frac{x_{i}}{1+2\lambda_{i} p_{i}}\right\|^{2}+\frac{1}{2\lambda}\sum_{i=1}^{N}\left\|\frac{x_{i}}{1+2\lambda_{i} p_{i}}-x_{i}\right\|^{2}\\
&=\sum_{i=1}^{N}\frac{p_{i}\omega_{i}\|x_{i}\|^{2}}{(1+2\lambda_{i} p_{i})^{2}}+\frac{1}{2\lambda}\frac{(2\lambda_{i} p_{i})^{2}\|x_{i}\|^{2}}{(1+2\lambda_{i} p_{i})^{2}}\\
&=\sum_{i=1}^{N}\frac{p_{i}\omega_{i}\|x_{i}\|^{2}+2\lambda\omega_{i}^{2}p_{i}^{2}\|x_{i}\|^{2}}{(1+2\lambda_{i} p_{i})^{2}}\\
&=\sum_{i=1}^{N}\frac{p_{i}\omega_{i}\|x_{i}\|^{2}(1+2\lambda\omega_{i}p_{i})}{(1+2\lambda_{i}p_{i})^{2}}\\
&=\sum_{i=1}^{N}\frac{p_{i}\omega_{i}\|x_{i}\|^{2}}{1+2\lambda_{i}p_{i}}.
\end{align*}
Observe that $p\mapsto \ell(p)=e_{\lambda}g_{p}(\bm{x})$ is concave since it is the infimum of concave functions.
\end{proof}
Note that if $\alpha_{i}=0$ for all $i\in\{1,\ldots,N\}$, then any $p\in\Delta_{N}$ is solution to \eqref{prob_aux_1}. Hence, we assume that $\displaystyle\max_{1\leq i\leq N}\alpha_{i}>0$. The following result provides an explicit solution to problem~\eqref{prob_aux_1}.
\begin{proposition}
\label{p_convex}
In the context of problem \eqref{prob_aux_1}, let $\{\ell_{i}\}_{i=1}^{N}$ such that $\alpha_{\ell_{1}}\leq\cdots\leq\alpha_{\ell_{N}}$ with $\alpha_{\ell_N}>0$ and let $I_{i}:=\{\ell_{1},\ldots,\ell_{i}\}$ for all $i\in\{1,\ldots,N\}$. Define
		\begin{align}
			\label{def_k}
			k:=\min\left\{i\in\{0,\ldots,N-1\}\,:\,(2+\sum_{j\notin I_{i}}1/\lambda_{j})\sqrt{\alpha_{\ell_{i+1}}}>\sum_{j\notin  I_{i}}\dfrac{\sqrt{\alpha_{j}}}{\lambda_{j}}\right\},
		\end{align}
		where ${I_{0}}:=\emptyset$. Then $\overline{p}\in \R^{N}$ defined by
		\begin{align}
			\label{sol_prob_paso_1}
   (\forall i\in\{1,\ldots,N\})\quad
			\overline{p}_{i}=
			\begin{cases}
				0, & \text{if}\,\, i\in {I_{k}};\\
				\dfrac{1}{2\lambda_{i}}\left[\dfrac{(2+\sum_{j\notin I_{k}}1/\lambda_{j})\sqrt{\alpha_{i}}}{\sum_{j\notin { I_{k}}}\frac{\sqrt{\alpha_{j}}}{\lambda_{j}}} -1\right],& \text{if}\,\, i\notin{ I_{k}}
			\end{cases}
		\end{align}
		is a solution to problem \eqref{prob_aux_1}.
	\end{proposition}
\begin{proof}
First of all, the set in \eqref{def_k} is nonempty since $N-1$ is in that set. Indeed, $(2+1/\lambda_{\ell_{N}})\sqrt{\alpha_{\ell_{N}}}>\sqrt{\alpha_{\ell_{N}}}/\lambda_{\ell_{N}}$. Hence $k$ is well defined. Note that $\nabla \ell(p)=\left(\dfrac{\alpha_{i}}{(1+2\lambda_{i} p_{i})^{2}}\right)_{i=1}^{N}$ for all $p\in\R^{N}$. Then, since $-\ell$ is convex on $\Delta_{N}$, by the definition of $\Delta_{N}$ and the KKT's conditions, it follows that it is enough to prove that there exists $\tau\in \R$ and $(\mu_{i})_{i=1}^{N}\in\R_{+}^{N}$ such that
		\begin{align}
			\label{cond_kkt}
			&\dfrac{-\alpha_{i}}{(1+2\lambda_{i} p_{i})^{2}} +\tau-\mu_{i}=0,\quad\mu_{i}p_{i}=0,\quad p_{i}\geq 0\quad\text{for all }i\in\{1,\ldots,N\},\\
			&\text{and }\sum_{i=1}^{N}p_{i}=1,\label{cond_kkt_2}
		\end{align}
		where $p\in\R^{N}$ is defined by \eqref{sol_prob_paso_1}. Consider
		\begin{align*}
			\tau=\left(\sum_{j\notin{ I_{k}}}\frac{\sqrt{\alpha_{j}}}{\lambda_{j}}\right)^{2}/\left(2+\sum_{j\notin I_{k}}1/\lambda_{j}\right)^{2}\in\R
		\end{align*}
		and $(\mu_{i})_{i=1}^{N}\in \R^{N}$ defined by $\mu_{i}=\tau-\alpha_{i}$ if $i\in { I_{k}}$ and $\mu_{i}=0$ if $i\notin { I_{k}}$. Thus, we have the second condition in \eqref{cond_kkt}. Let  {us} now prove the first condition in \eqref{cond_kkt}. Let $i\in\{1,\ldots,N\}$. If $i\in { I_{k}}$, then $p_{i}=0$ and $\dfrac{-\alpha_{i}}{(1+2\lambda_{i} p_{i})^{2}}+\tau-\mu_{i}=-\alpha_{i}+\tau-\mu_{i}=0$. If $i\notin { I_{k}}$, then by \eqref{sol_prob_paso_1}, we have that $(1+2\lambda_{i} p_{i})^{2}=\left(2+\sum_{j\notin I_{k}}1/\lambda_{j}\right)^{2}\alpha_{i}/\left(\sum_{j\notin I_{k}}\sqrt{\alpha_{j}}/\lambda_{j}\right)^{2}$ and hence
		\begin{align*}
			\frac{-\alpha_{i}}{(1+2\lambda_{i} p_{i})^{2}}+\tau-\mu_{i}=-\frac{\left(\sum_{j\notin { I_{k}}}\sqrt{\alpha_{j}}/\lambda_{j}\right)^{2}}{(2+\sum_{j\notin I_{k}}1/\lambda_{j})^{2}}+\tau-\mu_{i}=-\mu_{i}=0,
		\end{align*}
		which prove the first condition in \eqref{cond_kkt}.
		We claim that $\mu_{i}\geq 0$ for all $i\in  {I_{k}}$ (note that if $k=0$, the latter is direct since $ {I_{0}}=\emptyset$, so in order to prove this claim we assume that $k>0$). Let $i\in { I_{k}}$. Then $\alpha_{i}\leq \alpha_{\ell_{k}}$. Now, by definition of $k$, we have that $k-1$ is not in the set in \eqref{def_k}, that is,
		\begin{align*}
			(2+\sum_{j\notin I_{k-1}}1/\lambda_{j})\sqrt{\alpha_{\ell_{k}}}\leq \sum_{j\notin { I_{k-1}}}\sqrt{\alpha_{j}}/\lambda_{j} = \frac{\sqrt{\alpha_{\ell_{k}}}}{\lambda_{\ell_{k}}}+\cdots+\frac{\sqrt{\alpha_{\ell_{N}}}}{\lambda_{\ell_{N}}}.
		\end{align*}
		Therefore, the term $\frac{\sqrt{\alpha_{\ell_{k}}}}{\lambda_{\ell_{k}}}$ is canceled and it follows that 
\begin{align*}
(2+\sum_{j\notin I_{k}}1/\lambda_{j})\sqrt{\alpha_{i}}\leq (2+\sum_{j\notin I_{k}}1/\lambda_{j})\sqrt{\alpha_{\ell_{k}}}\leq \sum_{j\notin I_{k}}\sqrt{\alpha_{j}}/\lambda_{j},
\end{align*}
		which implies that $\mu_{i}=\tau-\alpha_{i}\geq 0$. We claim now that $p_{i}\geq 0$ for all $i\notin { I_{k}}$. Let $i\notin { I_{k}}$. Then $\alpha_{i}\geq \alpha_{\ell_{k+1}}$ and therefore
		\begin{align*}
			\dfrac{(2+\sum_{j\notin I_{k}}1/\lambda_{j})\sqrt{\alpha_{i}}}{\sum_{j\notin { I_{k}}}\sqrt{\alpha_{j}}/\lambda_{j}}\geq \dfrac{(2+\sum_{j\notin I_{k}}1/\lambda_{j})\sqrt{\alpha_{\ell_{k+1}}}}{\sum_{j\notin { I_{k}}}\sqrt{\alpha_{j}}/\lambda_{j}}>1,
		\end{align*}
		where the last inequality is by the definition of $k$. Thus, by definition in \eqref{sol_prob_paso_1}, we obtain that $p_{i}\geq 0$. Finally, we deduce from definition in \eqref{sol_prob_paso_1}, that 
        \begin{align*}
        \sum_{i=1}^{N}p_{i}=\dfrac{1}{2}\left(2+\sum_{j\notin I_{k}}\frac{1}{\lambda_{j}}-\sum_{i\notin I_{k}}\frac{1}{\lambda_{i}}\right)=1.
        \end{align*} 
        In summary, we have proved \eqref{cond_kkt}-\eqref{cond_kkt_2}.
\end{proof}
\subsection{Weakly convex case}
In this section, we derive a closed-form for the proximity operator of the following function
\begin{align}
\label{fun_nonconv_1}
(\forall \bm{x}\in\R^{nN})\quad f(\bm{x})=\max_{1\leq i\leq N}-\omega_{i}\|x_{i}\|^{2}.
\end{align}
Before we state the main result, we need some basic lemmas.
\begin{lemma}
\label{lema_1}
Let $I=\{1,\ldots,N\}$. For every $i\in I$, let $\rho_{i}>0$ and $f_{i}\in\Gamma_{\rho_{i}}(\R^{n})$. Let $\rho=\displaystyle\max_{i\in I}\rho_{i}$ and $f\colon\R^{nN}\rightarrow \R\cup\{+\infty\}$ defined by $f(\bm{x})=\sum_{i\in I}f_{i}(x_{i})$, for all $\bm{x}\in\R^{nN}$. Then $f\in\Gamma_{\rho}(\R^{nN})$ and, for every $\bm{x}\in\R^{nN}$ and $\lambda<1/\rho$, we have
\begin{align*}
\prox_{\lambda f}(\bm{x})=(\prox_{\lambda f_{i}}(x_{i}))_{i\in I}.
\end{align*}
\end{lemma}
\begin{proof}
Note that for every $i\in I$, we have
\begin{align*}
f_{i}+\dfrac{\rho}{2}\|\cdot\|^{2}=f_{i}+\dfrac{\rho_{i}}{2}\|\cdot\|^{2}+\dfrac{1}{2}(\rho-\rho_{i})\|\cdot\|^{2}.
\end{align*}
Then, $f_{i}\in \Gamma_{\rho}(\R^{n})$, for all $i\in I$. For every $i\in I$, let $\varphi_{i}=f_{i}+\dfrac{\rho}{2}\|\cdot\|^{2}\in\Gamma_{0}(\R^{n})$. Then
\begin{align*}
(\forall \bm{x}\in\R^{nN})\quad f(\bm{x})+\dfrac{\rho}{2}\|\bm{x}\|^{2}=\sum_{i\in I}\left(f_{i}(x_{i})+\dfrac{\rho}{2}\|x_{i}\|^{2}\right)=\sum_{i\in I}\varphi_{i}(x_{i})=:\varphi(\bm{x})
\end{align*}
Since $\varphi\in\Gamma_{0}(\R^{nN})$, then $f\in \Gamma_{\rho}(\R^{nN})$. Now, by \cite[Proposition~16.9]{MR3616647}, for every $\bm{x}\in\R^{nN}$ we have
\begin{align*}
\partial f(\bm{x})=\partial\varphi(\bm{x})-\rho\bm{x}=\bigtimes_{i\in I}(\partial\varphi_{i}(x_{i})-\rho x_{i})=\bigtimes_{i\in I}\partial f_{i}(x_{i}).
\end{align*}
Therefore, if $\lambda<1/\rho$, then
\begin{align*}
\bm{y}=\prox_{\lambda f}(\bm{x})&\Leftrightarrow \bm{x}\in \bm{y}+\lambda\partial f(\bm{y})\\
&\Leftrightarrow (\forall i\in I)\,\, x_{i}\in y_{i}+\lambda \partial f_{i}(x_{i})\\
&\Leftrightarrow (\forall i\in I)\,\, y_{i}=\prox_{\lambda f_{i}}(x_{i})\\
&\Leftrightarrow \bm{y}=(\prox_{\lambda f_{i}}(x_{i}))_{i\in I}.
\end{align*}
\end{proof}
\begin{lemma}
\label{lema_2}
Let $f\in\Gamma_{\rho}(\R^{n})$ and let $b\in\R^{n}$. Consider the function $g\colon\R^{n}\rightarrow\R\cup\{+\infty\}$ given by $g(x)=f(x-b)$, for all $x\in\R^{n}$. Then $g\in\Gamma_{\rho}(\R^{n})$ and for every $\lambda<1/\rho$, we have
\begin{align*}
(\forall x\in\R^{n})\quad\prox_{\lambda g}(x)=b+\prox_{\lambda f}(x-b).
\end{align*}
\end{lemma}
\begin{proof}
Let $x\in\R^{n}$. Then
\begin{align}
\label{w_fun_g}
g(x)+\dfrac{\rho}{2}\|x\|^{2}=f(x-b)+\dfrac{\rho}{2}\|x-b\|^{2}+\rho\langle x,b\rangle-\dfrac{\rho}{2}\|b\|^{2}.
\end{align}
Since $f\in\Gamma_{\rho}(\R^{n})$, then  $x\mapsto f(x-b)+\dfrac{\rho}{2}\|x-b\|^{2}$ belongs to $\Gamma_{0}(\R^{n})$. Thus, we conclude that $g\in\Gamma_{\rho}(\R^{n})$. Define $\varphi=f+\dfrac{\rho}{2}\|\cdot\|^{2}$ and $\widetilde{\varphi}\colon x\mapsto \varphi(x-b)$. From \eqref{w_fun_g}, for all $x\in\R^{n}$, we have
\begin{align*}
\partial g(x)&=\partial \left(g+\frac{\rho}{2}\|\cdot\|^{2}\right)(x)-\rho x\\
&=\partial \widetilde{\varphi}(x)+\rho b-\rho x\\
&=\partial \varphi(x-b)-\rho(x-b)=\partial f (x-b).
\end{align*}
Therefore, for every $x\in\R^{n}$ and $\lambda<1/\rho$, we have
\begin{align*}
y=\prox_{\lambda g}(x)&\Leftrightarrow x\in y+\lambda\partial g(y)\\
&\Leftrightarrow x\in y+\lambda\partial f(y-b)\\
&\Leftrightarrow x-b\in y-b+\lambda\partial f(y-b)\\
&\Leftrightarrow y-b=\prox_{\lambda f}(x-b)\Leftrightarrow y=b+\prox_{\lambda f}(x-b).
\end{align*}
\end{proof}
The following results provide a closed expression for the proximity operator of the function in \eqref{fun_nonconv_1}.
\begin{proposition}
\label{prox_w_convex}
Let $f$ the function in \eqref{fun_nonconv_1}. Then $f\in\Gamma_{\rho}(\R^{nN})$ with $\rho=2\displaystyle\max_{1\leq i\leq N}{\omega_{i}}$. Let $\bm{x}\in\R^{nN}$ and $\mu<1/\rho$. For every $i\in\{1,\ldots,N\}$, denote $\mu_{i}=\mu\omega_{i}$. Then
\begin{align*}
\prox_{\mu f}(\bm{x})=\left(\dfrac{x_{i}}{1-2\mu_{i}\overline{p}_{i}}\right)_{i=1}^{N},
\end{align*}
where $\overline{p}\in\R^{N}$ is a solution to
\begin{align}
\label{prob_aux_2}
\max_{p\in\Delta_{N}}\phi(p):=\sum_{i=1}^{N}\dfrac{\alpha_{i}p_{i}}{2\mu_{i}p_{i}-1},
\end{align}
and, for all $i\in\{1,\ldots,N\}$, $\alpha_{i}=\omega_{i}\|x_{i}\|^{2}$. 
\end{proposition}
\begin{proof}
Note that $f(\bm{x})=\displaystyle\max_{p\in\Delta_{N}}g_{p}(\bm{x})$, where $g_{p}(\bm{x})=\sum_{i=1}^{N}-p_{i}\omega_{i}\|x_{i}\|^{2}$ for all $p\in\Delta_{N}$. Given $p\in\Delta_{N}$ and $i\in\{1,\ldots,N\}$, define $f_{i}(z)=-p_{i}\omega_{i}\|z\|^{2}$, for all $z\in\R^{n}$. Clearly, for every $i\in\{1,\ldots,N\}$, $f_{i}\in\Gamma_{\rho_{i}}(\R^{n})$ with $\rho_{i}=2\omega_{i}$. Then, by Lemma~\ref{lema_1}, $g_{p}\in\Gamma_{\rho}(\R^{nN})$ for all $p\in\Delta_{N}$, with $\rho=2\max_{1\leq i\leq N}\omega_{i}$. Thus, by \cite[Lemma~5.1]{lopez2026projected}, $f\in\Gamma_{\rho}(\R^{nN})$. Now, the function $p\mapsto g_{p}(\bm{x})$ is concave and upper semicontinuous for all $\bm{x}\in\R^{nN}$. In addition, $f=\sup_{p\in\Delta_{N}}g_{p}$ is proper. Furthermore, $\Delta_{N}$ is a nonempty compact and convex set.
Then, by \cite[Theorem~5.2]{lopez2026projected}, for $\mu<1/\rho$, we have 
  \begin{align}
   \label{prox_sup_sep_2}
\prox_{\mu f}(\bm{x})=\prox_{\mu g_{\overline{p}}}(\bm{x}),\quad\text{ with } \overline{p}\in \displaystyle\argmax_{p\in\Delta_{N}}e_{\mu}g_{p}(\bm{x}).
  \end{align}
  Note that for all $i\in\{1,\ldots,N\}$ and $z\in\R^{n}$, we have
  \begin{align*}
y=\prox_{\mu f_{i}}(z)\Leftrightarrow z=y-2\mu p_{i}\omega_{i}y\Leftrightarrow y=\dfrac{z}{1-2\mu_{i}p_{i}}.
  \end{align*}
  Then, from Lemma~\ref{lema_1}, we have
  \begin{align}
  \label{prox_gp_2}
    \prox_{\mu g_{p}}(\bm{x})=\left(\dfrac{x_{i}}{1-2\mu_{i}p_{i}}\right)_{i=1}^{N}.
  \end{align}
  Let us compute $e_{\mu}g_{p}(\bm{x})$. By \eqref{prox_gp_2}, we have that
\begin{align*}
e_{\lambda}g_{p}(\bm{x})&=g_{p}(\prox_{\lambda g_{p}}(\bm{x}))+\frac{1}{2\lambda}\|\bm{x}-\prox_{\lambda g_{p}}(\bm{x})\|^{2}\\
&=\sum_{i=1}^{N}-p_{i}\omega_{i}\left\|\frac{x_{i}}{1-2\mu_{i} p_{i}}\right\|^{2}+\frac{1}{2\mu}\sum_{i=1}^{N}\left\|\frac{x_{i}}{1-2\mu_{i} p_{i}}-x_{i}\right\|^{2}\\
&=\sum_{i=1}^{N}-\frac{p_{i}\omega_{i}\|x_{i}\|^{2}}{(1-2\mu_{i} p_{i})^{2}}+\frac{1}{2\mu}\frac{(2\mu_{i} p_{i})^{2}\|x_{i}\|^{2}}{(1-2\mu_{i} p_{i})^{2}}\\
&=\sum_{i=1}^{N}\frac{-p_{i}\omega_{i}\|x_{i}\|^{2}+2\mu\omega_{i}^{2}p_{i}^{2}\|x_{i}\|^{2}}{(1-2\mu_{i} p_{i})^{2}}\\
&=\sum_{i=1}^{N}\frac{p_{i}\omega_{i}\|x_{i}\|^{2}(2\mu\omega_{i}p_{i}-1)}{(1-2\mu_{i}p_{i})^{2}}\\
&=\sum_{i=1}^{N}\frac{p_{i}\omega_{i}\|x_{i}\|^{2}}{2\mu_{i}p_{i}-1}.
\end{align*}
\end{proof}
In the context of problem~\eqref{prob_aux_2}, if $\alpha_{i}=0$ for some $i\in\{1,\ldots,N\}$, then a solution of problem~\eqref{prob_aux_2} is $p=e_{i}$, where $e_{i}\in\R^{N}$ is the canonical vector in $\R^{N}$. Then, we can assume that $\alpha_{i}>0$ for all $i\in\{1,\ldots,N\}$. We note that
\begin{align*}
(\nabla^{2}\phi(p))_{ij}=\begin{cases}
4\mu_{i}\alpha_{i}(2\mu_{i} p_{i}-1)^{-3} & \text{ if } i=j\\
0 & \text{ if }i\neq j
\end{cases}\quad\text{for all } i,j\in\{1,\ldots,N\}.
\end{align*}
Then, $\nabla^{2}(-\phi)(p)$ is a positive definite matrix for all $p\in\Delta_{N}$ since $\mu_{i}=\mu\omega_{i}<\omega_{i}/(2\max_{1\leq i\leq N}\omega_{i})\leq 1/2$. Thus $-\phi$ is strictly convex on $\Delta_{N}$. In addition, $\Delta_{N}$ is a compact convex set and $-\phi$ is continuous. Therefore, if $\alpha_{i}>0$ for all $i\in\{1,\ldots,N\}$, the problem~\eqref{prob_aux_2} has an unique solution.

The following result provides an explicit solution for the problem~\eqref{prob_aux_2}.
\begin{proposition}
		\label{prop_prob_aux_w}
		In the context of problem \eqref{prob_aux_2}, let $\{\ell_{i}\}_{i=1}^{N}$ such that $\alpha_{\ell_{1}}\geq\cdots\geq\alpha_{\ell_{N}}$ with $\alpha_{i}>0$ for all $i\in\{1,\ldots,N\}$ and let $I_{i}:=\{\ell_{1},\ldots,\ell_{i}\}$ for all $i\in\{1,\ldots,N\}$. Define
		\begin{align}
			\label{def_k_w}
			k:=\min\left\{i\in\{0,\ldots,N-1\}\,:\,(\sum_{j\notin I_{i}}1/\mu_{j}-2)\sqrt{\alpha_{\ell_{i+1}}}<\sum_{j\notin I_{i}}\frac{\sqrt{\alpha_{j}}}{\mu_{j}}\right\},
		\end{align}
		where $I_{0}:=\emptyset$. Then $\overline{p}\in \R^{N}$ defined by
		\begin{align}
			\label{sol_prob_aux_w}
			\overline{p}_{i}=
			\begin{cases}
				0 & \text{if}\,\, i\in I_{k}\\
				\dfrac{1}{2\mu_{i}}\left[1-\dfrac{(\sum_{j\notin I_{k}}1/\mu_{j}-2)\sqrt{\alpha_{i}}}{\sum_{j\notin I_{k}}\frac{\sqrt{\alpha_{j}}}{\mu_{j}}}\right]& \text{if}\,\, i\notin I_{k}
			\end{cases}\quad\text{ for all } i\in\{1,\ldots,N\}
		\end{align}
		is the solution to problem \eqref{prob_aux_2}.
	\end{proposition}
    \begin{proof}
Note that the set in \eqref{def_k_w} is nonempty since $N-1$ is in that set. Indeed, $(1/\mu_{\ell_{N}}-2)\sqrt{\alpha_{\ell_{N}}}<\sqrt{\alpha_{\ell_{N}}}/\mu_{\ell_{N}}$. Thus $k$ is well defined. Note that $\nabla \phi(p)=\left(\dfrac{-\alpha_{i}}{(2\mu_{i} p_{i}-1)^{2}}\right)_{i=1}^{N}$ for all $p\in\R^{N}$. Then, since $-\phi$ is convex on $\Delta_{N}$, by the definition of $\Delta_{N}$ and the KKT's conditions, it follows that it is enough to prove that there exists $\tau\in \R$ and $(\eta_{i})_{i=1}^{N}\in\R_{+}^{N}$ such that
		\begin{align}
			\label{cond_kkt_w}
			&\dfrac{\alpha_{i}}{(2\mu_{i} p_{i}-1)^{2}} +\tau-\eta_{i}=0 \quad\wedge\quad\eta_{i}p_{i}=0\quad\wedge\quad p_{i}\geq 0\quad\text{for all }i\in\{1,\ldots,N\},\\
			&\sum_{i=1}^{N}p_{i}=1,\label{cond_kkt_2_w}
		\end{align}
		where $p\in\R^{N}$ is defined by \eqref{sol_prob_aux_w}. Consider
		\begin{align*}
			\tau=-\left(\sum_{j\notin I_{k}}\frac{\sqrt{\alpha_{j}}}{\mu_{j}}\right)^{2}/\left(\sum_{j\notin I_{k}}1/\mu_{j} -2\right)^{2}\in\R
		\end{align*}
		and $(\eta_{i})_{i=1}^{N}\in \R^{N}$ defined by $\eta_{i}=\tau+\alpha_{i}$ if $i\in I_{k}$ and $\eta_{i}=0$ if $i\notin I_{k}$. Thus, we have the second condition in \eqref{cond_kkt_w}. Let us now prove the first condition in \eqref{cond_kkt_w}. Let $i\in\{1,\ldots,N\}$. If $i\in I_{k}$, then $p_{i}=0$ and $\dfrac{\alpha_{i}}{(2\mu_{i} p_{i}-1)^{2}}+\tau-\eta_{i}=\alpha_{i}+\tau-\eta_{i}=0$. If $i\notin I_{k}$, then by \eqref{sol_prob_aux_w}, we have that $(1-2\mu_{i} p_{i})^{2}=\left(\sum_{j\notin I_{k}}1/\mu_{j}-2\right)^{2}\alpha_{i}/\left(\sum_{j\notin I_{k}}\sqrt{\alpha_{j}}/\mu_{j}\right)^{2}$ and hence
		\begin{align*}
			\frac{\alpha_{i}}{(1-2\mu_{i} p_{i})^{2}}+\tau-\eta_{i}=\frac{\left(\sum_{j\notin I_{k}}\sqrt{\alpha_{j}}/\mu_{j}\right)^{2}}{(\sum_{j\notin I_{k}}1/\mu_{j}-2)^{2}}+\tau-\eta_{i}=-\eta_{i}=0,
		\end{align*}
		which prove the first condition in \eqref{cond_kkt_w}.
		We claim that $\eta_{i}\geq 0$ for all $i\in I_{k}$ (note that if $k=0$, the latter is direct since $I_{0}=\emptyset$, so in order to prove this claim we assume that $k>0$). Let $i\in I_{k}$. Then $\alpha_{i}\geq \alpha_{\ell_{k}}$. Now, by definition of $k$, we have that $k-1$ is not in the set in \eqref{def_k_w}, that is,
		\begin{align*}
			(\sum_{j\notin I_{k-1}}1/\mu_{j}-2)\sqrt{\alpha_{\ell_{k}}}\geq \sum_{j\notin I_{k-1}}\frac{\sqrt{\alpha_{j}}}{\mu_{j}} = \frac{\sqrt{\alpha_{\ell_{k}}}}{\mu_{\ell_{k}}}+\cdots+\frac{\sqrt{\alpha_{\ell_{N}}}}{\mu_{\ell_{N}}},
		\end{align*}
		Hence the term $\frac{\sqrt{\alpha_{\ell_{k}}}}{\mu_{\ell_{k}}}$ is canceled, which yields that
		\begin{align*}
			(\sum_{j\notin I_{k}}1/\mu_{j}-2)\sqrt{\alpha_{i}}\geq (\sum_{j\notin I_{k}}1/\mu_{j}-2)\sqrt{\alpha_{\ell_{k}}}\geq\sum_{j\notin I_{k}}\frac{\sqrt{\alpha_{j}}}{\mu_{j}},
		\end{align*}
		which implies that $\eta_{i}=\tau+\alpha_{i}\geq 0$. We claim now that $p_{i}\geq 0$ for all $i\notin I_{k}$. Let $i\notin I_{k}$. Then $\alpha_{i}\leq \alpha_{\ell_{k+1}}$ and therefore
		\begin{align*}
			\dfrac{(\sum_{j\notin I_{k}}1/\mu_{j} -2)\sqrt{\alpha_{i}}}{\sum_{j\notin I_{k}}\sqrt{\alpha_{j}}/\mu_{j}}\leq \dfrac{(\sum_{j\notin I_{k}}1/\mu_{j} -2)\sqrt{\alpha_{\ell_{k+1}}}}{\sum_{j\notin I_{k}}\sqrt{\alpha_{j}}/\mu_{j}}<1,
		\end{align*}
		where the last inequality is by the definition of $k$. Thus, by definition in \eqref{sol_prob_aux_w}, we obtain that $p_{i}\geq 0$. Finally, we deduce from definition in \eqref{sol_prob_aux_w}, that
        \begin{align*}
        \sum_{i=1}^{N}p_{i}=\dfrac{1}{2}\left(\sum_{i\notin I_{k}}\frac{1}{\mu_{i}}-\left(\sum_{j\notin I_{k}}\frac{1}{\mu_{j}}-2\right)\right)=1.
        \end{align*}
    Thus, we have proved \eqref{cond_kkt_w}-\eqref{cond_kkt_2_w}.
 \end{proof}

 \section{Applications}
\subsection{Denoising}
  Consider a known noisy measurement of a signal $x\in\R^{n}$:
\begin{align*}
b=x+w,
\end{align*}
where $x$ is an unknown signal and $w$ is an unknown noise vector. Given the vector $b$, the objective of the denoising problem is to find a good estimate of $x$. The denoisgin problem has associated the following least squares optimization problem
\begin{align}
\label{den_reg_1}
\min_{x\in\R^{n}} \|x-b\|^{2}+\lambda R(x),
\end{align}
where $R(x)$ is a regularization function which represents some a priori information on the signal $x$ and $\lambda>0$ is a regularization parameter. Usually, the regularization function has the following form
\begin{align*}
R(x)=\sum_{i=1}^{n-1}(x_{i}-x_{i+1})^{2}.
\end{align*}
Note that $R(x)=\|Lx\|^{2}$ for some $L\in\R^{(n-1)\times n}$. In this case, by the Fermat's rule, we obtain that the solution to \eqref{den_reg_1} is given by
\begin{align*}
\overline{x}=(\Id+\lambda L^{\top}L)^{-1}b.
\end{align*}
Consider now $N$ known noisy measurement $b^{1},\ldots,b^{N}$:
\begin{align*}
b^{i}=x^{i}+u^{i},
\end{align*}
where, for every $i\in\{1,\ldots,N\}$, $x^{i}\in\R^{n}$ is the unknown signal and $w^{i}$ is the unknown noise vector. We consider the following problem of finding a good estimate of $x^{i}$ for all $i\in\{1,\ldots,N\}$ in order to minimize the largest norm of the noise vectors with a weight $\omega_{i}>0$ associated to each component $u^{i}$. That is, we consider the following robust optimization problem
\begin{align}
\label{dro_quad_sep}
\min_{\bm{x}=(x^{1},\ldots,x^{N})\in\R^{nN}}\left\{\lambda\sum_{j=1}^{N}\|L_{j}x^{j}\|^{2}+\max_{1\leq i\leq N}\omega_{i}\|x^{i}-b^{i}\|^{2}\right\},
\end{align}
where, for every $j\in\{1,\ldots,N\}$, the quadratic term $\|L_{j}x^{j}\|^{2}$ represents the regularization term of the variable $x^{j}$. The problem~\eqref{dro_quad_sep} is equivalent to
\begin{align*}
\min_{\bm{x}\in\R^{nN}} H(\bm{x})+f(\bm{x}),
\end{align*}
where $H(\bm{x})=\lambda\sum_{j=1}^{N}\|L_{j}x^{j}\|^{2}$ and $f(\bm{x})=\displaystyle\max_{1\leq i\leq N}\omega_{i}\|x^{i}-b^{i}\|^{2}$. Note that 
\begin{align*}
\nabla H(\bm{x})=2\lambda(L_{1}^{\top}L_{1}x^{1},\ldots,L_{N}^{\top}L_{N}x^{N}).
\end{align*}
Thus, for every $\bm{x}$, $\bm{y}\in\R^{nN}$, we have
\begin{align*}
\|\nabla H(\bm{x})-\nabla H(\bm{y})\|^{2}&=\sum_{j=1}
^{N}\|2\lambda L_{j}^{\top}L_{j}(x^{j}-y^{j})\|^{2}\\
&\leq (2\lambda)^{2}\max_{1\leq j\leq N}\|L_{j}^{\top}L_{j}\|^{2}\|\bm{x}-\bm{y}\|^{2}. 
\end{align*}
Hence $\nabla H$ is $\beta^{-1}$-Lipschitz with $\beta^{-1}=2\lambda\displaystyle\max_{1\leq j\leq N}\|L_{j}^{\top}L_{j}\|$. On the other hand, $f(\bm{x})=g(\bm{x}-\bm{b})$, where $g(\bm{x})=\displaystyle\max_{1\leq i\leq N}\omega_{i}\|x^{i}\|^{2}$ and $\bm{b}=(b^{1},\ldots,b^{N})$. That is, $g$ is the function in \eqref{fun_conv_1}. Then, by \cite[Proposition~24.8(ii)]{MR3616647}, we have
\begin{align*}
\prox_{\lambda f}(\bm{x})=\bm{b}+\prox_{\lambda g}(\bm{x}-\bm{b}).
\end{align*}
 The proximity operator of $g$ is given by Propositions~\ref{prox_conv_1} and \ref{p_convex}. Note that $f$ is coercive. Then, problem~\eqref{dro_quad_sep} has solutions. For $\bm{\alpha}=(\alpha_{i})_{i=1}^{N}\in\R_{+}^{N}$ and $\bm{\lambda}=(\lambda_{i})_{i=1}^{N}\in\R_{++}^{N}$, denote by $\overline{p}(\bm{\alpha},\bm{\lambda})$ the solution to \eqref{prob_aux_1}, which has explicit form given by Proposition~\ref{p_convex}. Thus, we can apply the Forward-Backward algorithm \cite[Corollary~28.9]{MR3616647} to obtain the following convergence result.
 \begin{proposition}
    Let $\lambda\in\left]0,2\beta\right[$ and let $\bm{x}^{0}\in\R^{nN}$. For every $n\in\mathbb{N}$, consider
   \begin{equation}
\label{alg_fb_1}
\left\lfloor 
	\begin{array}{ll}
\bm{y}^{n}=\bm{x}^{n}-\lambda\nabla H(\bm{x}^{n})\\
p^{n}=\overline{p}((\omega_{i}\|y_{i}^{n}-b^{i}\|^{2})_{i=1}^{N},(\lambda\omega_{i})_{i=1}^{N})\\
\bm{x}^{n+1}=\bm{b}+\left(\dfrac{y_{i}^{n}-b^{i}}{1+2\lambda\omega_{i}p_{i}^{n}}\right)_{i=1}^{N}.
\end{array}
	\right. 
\end{equation}
Then $\bm{x}^{n}\rightarrow \overline{\bm{x}}$, where $ \overline{\bm{x}}$ is solution to \eqref{dro_quad_sep}.
 \end{proposition}
 \subsection{Max-min dispersion problem}
  In this subsection, we consider the following weighted max dispersion problem (or weighted max-min location problem):
 \begin{align}
 \label{max_dis}
\max_{x\in S}\min_{1\leq i\leq N}\omega_{i}\|x-u_{i}\|^{2},
 \end{align}
 where $S\subset\R^{n}$ is a nonempty compact convex subset and for every $i\in\{1,\ldots,N\}$, $u_{i}\in\R^{n}$ and $\omega_{i}>0$. The objective of problem~\eqref{max_dis} is to find a point in $S$ that is furthest from a given set of points $u_{1},\ldots,u_{N}$ in a weighted maxmin sense. When $\omega_{1}=\ldots=\omega_{N}$, the problem has the geometric interpretation of finding the largest Euclidean sphere with center in $S$ and enclosing no given point. Since $x\mapsto \min_{1\leq i\leq N}\|x-u_{i}\|^{2}$ is continuous, the problem~\eqref{max_dis} has solutions. The problem \eqref{max_dis} have been studied in \cite{jeyakumar2018exact} when the constraint $S$ is a polyhedral or a ball, and has applications in facility location, pattern recognition, among others \cite{dasarathy1980maxmin,johnson1990minimax}. In the general case, the weighted maxmin dispersion problem is known to be NP-hard \cite{ravi1994heuristic} even when $S$ is a box and the weights are equal \cite{haines2013convex}.

  In order to solve \eqref{max_dis}, we use the penalty method described in \cite[Section~A.2]{lopez2026projected}. That is, we solve the following penalized problem
 \begin{align}
 \label{pen_max_dis}
\min_{x\in \R^{n}}\lambda P(x)+\max_{1\leq i\leq N}-\omega_{i}\|x-u_{i}\|^{2},
 \end{align}
     where $P$ is the penalty function which is lower semicontinuous and satisfies that $P(x)\geq 0$ for all $x\in\R^{n}$ and $P(x)=0$ if and only if $x\in S$, and $\lambda>0$ is the penalty parameter. Consider the penalty function $P(x)=\frac{1}{2}d^{2}(x,S)$. Note that an equivalent formulation for the problem~\eqref{pen_max_dis} is the following
 \begin{align}
 \label{app_w_2}
\min_{\bm{x}=(x_{1},\ldots,x_{N})\in\R^{nN}} H(\bm{x})+g(\bm{x})\quad\text{s.t.}\quad\bm{x}\in \mathcal{D},
 \end{align}
 where $H\colon\R^{nN}\rightarrow\R$ is defined by $H(\bm{x})=\frac{\lambda}{2}d^{2}(x_{1},S)$, $\mathcal{D}=\{\bm{x}\in\R^{nN}:x_{1}=\cdots=x_{N}\}$ is the diagonal set, and $g$ is given by
 \begin{align}
\label{fun_g_1}
(\forall\bm{x}\in\R^{nN})\quad g(\bm{x})=\max_{1\leq i\leq N}-\omega_{i}\|x_{i}-u_{i}\|^{2}.
 \end{align}
 Observe that $\nabla H(\bm{x})=(\lambda(x_{1}-P_{S}(x_{1})),0,\ldots,0)$ is Lipschitz with constant $\lambda$. On the other hand, note that $g(\bm{x})=f(\bm{x}-\bm{u})$ for all $\bm{x}\in\R^{nN}$, where $\bm{u}=(u_{1},\ldots,u_{N})\in\R^{nN}$ and $f$ is given by \eqref{fun_nonconv_1}. Then, by Proposition~\ref{prox_w_convex} and Lemma~\ref{lema_2}, $g\in\Gamma_{\rho}(\R^{nN})$ with $\rho=2\max_{1\leq i\leq N}\omega_{i}$. Moreover, for all $\mu<1/\rho$, we have
 \begin{align*}
(\forall \bm{x}\in\R^{nN})\quad\prox_{\mu g}(\bm{x})=\bm{u}+\prox_{\mu f}(\bm{x}-\bm{u}).
 \end{align*}
The proximity operator of $f$ is given by Propositions~\ref{prox_w_convex} and \ref{prop_prob_aux_w}. To solve \eqref{app_w_2}, we use the algorithm in \cite{lopez2026projected}. From \cite[Theorem~3.2]{lopez2026projected} we obtain the following result.
\begin{theorem}
Let $\bm{x}^{1}\in\mathcal{D}$, $\alpha\in\left]0,1\right[$, and $C\in\left]0,1/4\right]$. For every $k\geq 1$, consider
\begin{equation}
\label{alg_proj_w_1}
\left\lfloor 
	\begin{array}{ll}
\mu_{k}=Ck^{-\alpha}\\
\bm{z}^{k}=(\lambda(x_{1}^{k}-P_{S}(x_{1}^{k})),0,\ldots,0)\\
\bm{x}^{k+1}=P_{\mathcal{D}}\left(\dfrac{1}{1+\lambda\mu_{k}}\left(\mu_{k}(\lambda \bm{x}^{k}-\bm{z}^{k})+\prox_{\mu_{k}g}(\bm{x}^{k})\right)\right).
\end{array}
	\right. 
\end{equation}
 For every $\bm{x}\in\R^{nN}$ and $k\geq 1$, define $F_{k}(\bm{x})=H(\bm{x})+e_{\mu_{k}}g(\bm{x})$. Assume that the sequence $(\bm{x}^{k})$ generated by the Algorithm~\ref{alg_proj_w_1} is bounded and assume that the sequence $(F_{k}(\bm{x}^{k}))$ is bounded from below by a real number $F^{*}$. Then, there exist $\ell\geq 0$ and $N_{0}\geq 1$ such that 
\begin{equation*}
\begin{aligned}
\min_{N_{0}\leq j\leq k}d(-\nabla H(\bm{x}^{j}), \partial g(\prox_{\mu_{j}g}(\bm{x}^{j}))+N_{\mathcal{D}}(\bm{x}^{j}))&\leq k^{\frac{\alpha-1}{2}}\widetilde{C}_{N_{0}} \textrm{ for all } k\geq N_{0},\\
\Vert \bm{x}^{k}-\prox_{\mu_{k}g}(\bm{x}^{k})\Vert &\leq k^{-\alpha}C \ell  \textrm{ for all } k\geq N_{0},
\end{aligned}
\end{equation*}
where $\widetilde{C}_{N_{0}}=\frac{\sqrt{2}\sqrt{\lambda+C^{-1}}}{\sqrt{\left(1+1/N_{0}\right)^{1-\alpha}-1}}\sqrt{F_{N_{0}}(\bm{x}^{N_{0}})-F^{*}+\mu_{N_{0}} \ell^{2}}$.
\end{theorem}
Note that the projection onto $\mathcal{D}$ is given by \cite[Proposition~29.16]{MR3616647}:
\begin{align}
\label{proj_inter_1}
(\forall\bm{x}\in\R^{nN})\quad P_{\mathcal{D}}(\bm{x})=\left(\dfrac{1}{N}\sum_{i=1}^{N}x_{i}\right)_{j=1}^{N}.
\end{align}
When $\lambda>0$ is large enough, the previous theorem allows to find an approximated critical point of problem~\eqref{max_dis}.
\\

\small{{\bf{Acknowledgements}} The author was supported by ANID Chile under grant Fondecyt Postdoctorado N°3260185.}

\end{document}